\documentclass[10pt,reqno]{amsart}

\usepackage[T1]{fontenc}
\usepackage[utf8]{inputenc}
\usepackage{amsmath,amssymb,mathtools,mathrsfs}
\IfFileExists{newtxtext.sty}
  {\usepackage{newtxtext,newtxmath}}
  {\usepackage{lmodern}}
\usepackage[a4paper,margin=1.15in]{geometry}
\usepackage{microtype}
\usepackage{xcolor}
\usepackage[
  colorlinks=true,
  linkcolor=blue,
  citecolor=blue,
  urlcolor=blue
]{hyperref}

\allowdisplaybreaks
\numberwithin{equation}{section}

\newcommand{\R}{\mathbb R}
\newcommand{\nalg}{\mathfrak n}
\newcommand{\zalg}{\mathfrak z}
\newcommand{\valg}{\mathfrak v}
\newcommand{\aalg}{\mathfrak a}
\newcommand{\RicB}{\Ric^{B}}
\DeclareMathOperator{\Ric}{Ric}
\DeclareMathOperator{\Aut}{Aut}
\DeclareMathOperator{\Der}{Der}
\DeclareMathOperator{\ad}{ad}
\DeclareMathOperator{\Sym}{Sym}
\DeclareMathOperator{\tr}{tr}
\DeclareMathOperator{\im}{im}

\DeclareMathOperator{\Id}{Id}

\theoremstyle{plain}
\newtheorem{theorem}{Theorem}[section]
\newtheorem{proposition}[theorem]{Proposition}
\newtheorem{lemma}[theorem]{Lemma}

\theoremstyle{remark}
\newtheorem{remark}[theorem]{Remark}
\newtheorem{example}[theorem]{Example}

\title[Harmonic Torsion of Generalized Nilsolitons]
{Harmonic Torsion of Generalized Nilsolitons on Nilpotent Lie Groups}

\author{Huibin Chen}
\address{Institute of Mathematics, School of Mathematical Sciences,
Nanjing Normal University, Nanjing 210023, P.R. China}
\email{chenhuibin@njnu.edu.cn}

\author{Zhiqi Chen}
\address{School of Mathematics and Statistics, Guangdong University of
Technology, Guangzhou 510520, P.R. China}
\email{chenzhiqi@gdut.edu.cn}

\author{Fuhai Zhu}
\address{Department of Mathematics, Nanjing University, Nanjing 210023,
P.R. China}
\email{zhufuhai@nju.edu.cn}

\date{}

\hypersetup{
  pdftitle={Harmonic Torsion of Generalized Nilsolitons on Nilpotent Lie Groups},
  pdfauthor={Huibin Chen, Zhiqi Chen, Fuhai Zhu},
  pdfsubject={Generalized nilsolitons and harmonic torsion},
  pdfkeywords={generalized Ricci flow, generalized nilsoliton, nilpotent Lie group, harmonic torsion}
}

\begin{document}

\begin{abstract}
We prove that the torsion three-form in the preferred splitting of a
left-invariant generalized nilsoliton on a nilpotent Lie group is
harmonic.  This gives a positive answer to Question~7.12 proposed by Fusi--Lafuente--Stanfield in \cite{FLS2024}. 
\end{abstract}

\maketitle

\section{Introduction}\label{sec:introduction}

The generalized Ricci flow is a coupled evolution equation for a
Riemannian metric and a closed three-form.  In an isotropic splitting of an
exact Courant algebroid, it can be written as
\begin{equation*}
 \frac{\partial g}{\partial t}=-2\Ric_g+\frac12H^2,
 \qquad
 \frac{\partial H}{\partial t}=\Delta_gH,
\end{equation*}
where
\begin{equation*}
 H^2(X,Y)=g(\iota_XH,\iota_YH),
 \qquad
 \Delta_g=-(dd_g^*+d_g^*d).
\end{equation*}
The symmetric tensor in the metric equation is
$\RicB_{g,H}=\Ric_g-\frac14H^2$.  We refer to
\cite{GFS2021,Gualtieri2011} for the background on generalized metrics and
exact Courant algebroids.

Fusi, Lafuente and Stanfield developed a bracket-flow approach to the
homogeneous generalized Ricci flow and introduced algebraic generalized
Ricci solitons in \cite{FLS2024}.  When the underlying Lie algebra is
nilpotent, these solitons are called \emph{generalized nilsolitons}.  They
extend classical nilsolitons \cite{Lauret2001} and include examples with
nonzero torsion.  The generalized Ricci flow on nilpotent Lie groups was
also studied in \cite{Paradiso2021}.

Let $E$ be a left-invariant exact Courant algebroid over a nilpotent Lie
group $N$, and let $\mathcal G$ be a left-invariant generalized metric.
The preferred isotropic splitting determined by $\mathcal G$ gives an
identification
\begin{equation*}
 (E,\mathcal G)\simeq
 \bigl((\nalg\oplus\nalg^*)_H,\mathcal G(\bar g,0)\bigr),
\end{equation*}
where $\nalg$ is the Lie algebra of $N$ and
$H\in\Lambda^3\nalg^*$ is closed.  In this paper, harmonicity always means
harmonicity in the left-invariant Chevalley--Eilenberg complex, namely,
$d_\mu H=d_\mu^*H=0$.  It does not refer to global $L^2$-harmonicity on the
noncompact group.

In \cite[Question~7.12]{FLS2024}, Fusi, Lafuente and Stanfield asked
whether the preferred torsion of a generalized nilsoliton is necessarily
harmonic.  Our main result answers this question affirmatively for
nilpotent Lie groups.  All references to numbered statements and
equations in \cite{FLS2024} are to arXiv version~1.

\begin{theorem}\label{thm:main}
Let $E$ be a left-invariant exact Courant algebroid over a nilpotent Lie
group $N$, and let $\mathcal G$ be a left-invariant generalized
nilsoliton.  In the preferred splitting
\begin{equation*}
 (E,\mathcal G)\simeq
 \bigl((\nalg\oplus\nalg^*)_H,\mathcal G(\bar g,0)\bigr),
\end{equation*}
the closed three-form $H$ is $\bar g$-harmonic.
\end{theorem}

In Section~\ref{sec:generalized-nilsolitons}, we recall the algebraic
soliton equations and derive the identity satisfied by $d_\mu^*H$.  The
Hodge--Ricci inequality is proved in Section~\ref{sec:hodge-ricci}.
Theorem~\ref{thm:main} and Corollary~\ref{cor:uniqueness} are proved in
Section~\ref{sec:harmonicity}.  Section~\ref{sec:refinements} contains the
two-step and almost-abelian identities, together with an example showing
that closed invariant three-forms need not be coclosed.

\section{Generalized nilsolitons and the codifferential identity}\label{sec:generalized-nilsolitons}

Throughout the paper, we follow the conventions of \cite{FLS2024}.  Fix a
$\bar g$-orthonormal basis $\{e_i\}$ of $\nalg$.  The metric $\bar g$ is
extended to alternating tensors by full contraction:
for $\sigma,\tau\in\Lambda^k\nalg^*$,
\begin{equation}\label{eq:FLS-form-pairing}
 \bar g(\sigma,\tau)
 =\sum_{i_1,\ldots,i_k}
 \sigma(e_{i_1},\ldots,e_{i_k})
 \tau(e_{i_1},\ldots,e_{i_k}),
 \qquad |\sigma|^2=\bar g(\sigma,\sigma).
\end{equation}
The operator $d_\mu^*$ is the Riemannian codifferential.  Since a
nilpotent Lie algebra is unimodular, it agrees on left-invariant forms
with the codifferential of the Chevalley--Eilenberg complex.
For $\alpha\in\Lambda^k\nalg^*$ and
$\beta\in\Lambda^{k+1}\nalg^*$, the convention
\eqref{eq:FLS-form-pairing} gives
\begin{equation}\label{eq:codifferential-pairing}
 \bar g(d_\mu^*\beta,\alpha)
 =\frac1{k+1}\bar g(\beta,d_\mu\alpha).
\end{equation}

For a two-form $\alpha$ and a three-form $H$, we use the notation of
\cite{FLS2024} and define the corresponding self-adjoint endomorphisms by
\begin{equation}\label{eq:H-square}
 \bar g(\alpha^2X,Y)=\bar g(\iota_X\alpha,\iota_Y\alpha),
 \qquad
 \bar g(H^2X,Y)=\bar g(\iota_XH,\iota_YH).
\end{equation}
Thus the Bismut--Ricci endomorphism is
\begin{equation}\label{eq:RicB}
 \RicB_{\mu,H}=\Ric_\mu-\frac14H^2.
\end{equation}

For $A\in\mathfrak{gl}(\nalg)$, let $\rho(A)$ denote the induced
representation on forms:
\begin{equation*}
 (\rho(A)\sigma)(X_1,\ldots,X_k)
 =-\sum_{r=1}^k
 \sigma(X_1,\ldots,AX_r,\ldots,X_k).
\end{equation*}
Thus $\rho(\Id)\sigma=-k\sigma$ for a $k$-form $\sigma$.

\subsection{The algebraic soliton equations}\label{subsec:soliton-equations}

By \cite[Proposition~7.4]{FLS2024}, an algebraic generalized Ricci
soliton is determined by a closed three-form $H$, a number
$\lambda\in\R$, and a derivation $D\in\Der(\nalg)$ satisfying
\begin{equation}\label{eq:soliton-system}
 \RicB_{\mu,H}=\lambda\Id+D,
 \qquad
 \Delta_\mu H=\lambda H+\rho(\RicB_{\mu,H})H
          =-2\lambda H+\rho(D)H.
\end{equation}
Here $\Delta_\mu=-(d_\mu d_\mu^*+d_\mu^*d_\mu)$.  Since
$\RicB_{\mu,H}$ is symmetric, the
derivation $D$ is symmetric.  When $\nalg$ is nilpotent, such a soliton is
called a generalized nilsoliton.

\subsection{The codifferential identity}\label{subsec:codifferential-identity}

We first record the relation between the representation on two-forms and
the symmetric endomorphism $\alpha^2$.

\begin{lemma}\label{lem:representation-contraction}
Let $A\in\Sym(\nalg)$ and $\alpha\in\Lambda^2\nalg^*$.  Then
\begin{equation}\label{eq:representation-contraction}
 \bar g(\rho(A)\alpha,\alpha)=-2\tr(\alpha^2A).
\end{equation}
\end{lemma}

\begin{proof}
Choose an orthonormal basis such that $Ae_i=a_i e_i$ and write
$\alpha=\sum_{i<j}\alpha_{ij}e^{ij}$.  Then
\begin{equation*}
 \bar g(\rho(A)\alpha,\alpha)
 =-2\sum_{i<j}(a_i+a_j)\alpha_{ij}^2.
\end{equation*}
On the other hand,
$\bar g(\alpha^2e_i,e_i)=\sum_{j\ne i}\alpha_{ij}^2$.  Taking the
trace of $\alpha^2A$ proves \eqref{eq:representation-contraction}.
\end{proof}

\begin{proposition}\label{prop:codifferential-identity}
Let $(\mu,H)$ be an algebraic generalized Ricci soliton and set
$\eta=d_\mu^*H$.  Then
\begin{equation}\label{eq:codifferential-operator}
 \rho(\RicB_{\mu,H})\eta+d_\mu^*d_\mu\eta=0.
\end{equation}
Consequently,
\begin{equation}\label{eq:codifferential-energy}
 0=\frac16|d_\mu\eta|^2-\tr(\eta^2\RicB_{\mu,H}).
\end{equation}
\end{proposition}

\begin{proof}
Set $R=\RicB_{\mu,H}$.  Since $R=\lambda\Id+D$ and $D$ is a
derivation, the induced action on the Lie bracket $\mu$ satisfies
$\theta(R)\mu=-\lambda\mu$.  The variation formula
\cite[equation~(36)]{FLS2024} is
\begin{equation*}
 d^*_{\theta(A)\mu}=[d^*_{\mu},\rho(A^t)].
\end{equation*}
As $R^t=R$, it follows that
\begin{equation*}
 -\lambda d_\mu^*=d_\mu^*\rho(R)-\rho(R)d_\mu^*.
\end{equation*}

Since $d_\mu H=0$, the second equation in \eqref{eq:soliton-system}
becomes $\rho(R)H+d_\mu d_\mu^*H=-\lambda H$.  Applying $d_\mu^*$ and
using the preceding
commutator identity gives
\begin{equation*}
 \rho(R)d_\mu^*H+d_\mu^*d_\mu d_\mu^*H=0.
\end{equation*}
This is \eqref{eq:codifferential-operator}; it is equation~(53) in
\cite{FLS2024}.

Taking the $\bar g$-inner product with $\eta$ and using
Lemma~\ref{lem:representation-contraction}, we obtain
\begin{equation*}
 0=\bar g(\rho(R)\eta,\eta)
   +\bar g(d_\mu^*d_\mu\eta,\eta)
  =-2\tr(\eta^2R)+\frac13|d_\mu\eta|^2.
\end{equation*}
Dividing by two proves \eqref{eq:codifferential-energy}.
\end{proof}

\section{A Hodge--Ricci inequality on nilpotent Lie algebras}\label{sec:hodge-ricci}

We prove Theorem~\ref{thm:hodge-ricci}.  The main algebraic lemma does not
require the Jacobi identity.  Nilpotency enters only when the moment-map
term is identified with the Ricci operator.

Let $(V,\bar g)$ be a Euclidean vector space and let
$\mu\in\Lambda^2V^*\otimes V$.  For $\alpha\in\Lambda^2V^*$, set
\begin{equation*}
\begin{split}
 (d_\mu\alpha)(X,Y,Z)={}&-\alpha(\mu(X,Y),Z)
 +\alpha(\mu(X,Z),Y)\\
 &-\alpha(\mu(Y,Z),X).
\end{split}
\end{equation*}
Define $M_\mu\in\Sym(V)$ by
\begin{equation}\label{eq:moment-map-Ricci}
\begin{split}
 \bar g(M_\mu X,Y)={}&-\frac12\sum_i
 \bar g(\mu(X,e_i),\mu(Y,e_i))\\
 &+\frac14\sum_{i,j}
 \bar g(\mu(e_i,e_j),X)\bar g(\mu(e_i,e_j),Y),
\end{split}
\end{equation}
where $\{e_i\}$ is an orthonormal basis of $V$.

\begin{lemma}\label{lem:algebraic-hodge-ricci}
With the notation above, for every $\alpha\in\Lambda^2V^*$,
\begin{equation}\label{eq:algebraic-hodge-ricci}
 \frac16|d_\mu\alpha|^2-2\tr(\alpha^2M_\mu)\geq0.
\end{equation}
\end{lemma}

\begin{proof}
Write $K=K_\alpha$.  Since $\alpha^2=K^*K$, the cyclicity of the trace
and \eqref{eq:moment-map-Ricci} give
\begin{equation}\label{eq:trace-K-mu}
 \tr(\alpha^2M_\mu)=-\frac12\mathsf U+\frac14\mathsf V,
\end{equation}
where
\begin{equation*}
 \mathsf U=\sum_{a,i}|\mu(Ke_a,e_i)|^2,
 \qquad
 \mathsf V=\sum_{i,j}|K\mu(e_i,e_j)|^2.
\end{equation*}

We use the orthogonal normal form of $K$.  After changing the orthonormal
basis, there exist $m\geq0$ and numbers
$\kappa_1,\ldots,\kappa_m>0$ such that
\begin{equation*}
 Ke_{2s-1}=\kappa_s e_{2s},
 \qquad
 Ke_{2s}=-\kappa_s e_{2s-1},
 \qquad 1\leq s\leq m,
\end{equation*}
and $Ke_r=0$ for $r>2m$.  Define an involution $\tau$ by
\begin{equation*}
 \tau(2s-1)=2s,
 \qquad
 \tau(2s)=2s-1,
 \qquad
 \tau(r)=r\quad\text{if }r>2m.
\end{equation*}
We also set
\begin{equation*}
\begin{split}
 &\kappa_{2s-1}=\kappa_{2s}=\kappa_s,
 \qquad
 \varepsilon_{2s-1}=1,
 \qquad
 \varepsilon_{2s}=-1,\\
 &\kappa_r=0,
 \qquad
 \varepsilon_r=1\quad\text{if }r>2m.
\end{split}
\end{equation*}
Thus $Ke_i=\varepsilon_i\kappa_i e_{\tau(i)}$ for every $i$.
On every nonzero block,
$\varepsilon_{\tau(i)}=-\varepsilon_i$; no such relation is used on
$\ker K$, where $\kappa_i=0$.

Write $c_{ij}^{\,k}=\bar g(\mu(e_i,e_j),e_k)$.  Since the first two
indices in $\mathsf V$ are ordered, we have
\begin{equation*}
 \mathsf U=\sum_{i<j,k}
 (\kappa_i^2+\kappa_j^2)(c_{ij}^{\,k})^2,
 \qquad
 \mathsf V=2\sum_{i<j,k}\kappa_k^2(c_{ij}^{\,k})^2.
\end{equation*}
It follows from \eqref{eq:trace-K-mu} that
\begin{equation}\label{eq:pre-Darboux-SOS}
\begin{split}
 &\frac16|d_\mu\alpha|^2-2\tr(\alpha^2M_\mu)\\
 &\qquad=\frac16|d_\mu\alpha|^2
 +\sum_{i<j,k}
 (\kappa_i^2+\kappa_j^2-\kappa_k^2)(c_{ij}^{\,k})^2.
\end{split}
\end{equation}

Fix $i<j<\ell$ and put
$a=\kappa_i$, $b=\kappa_j$ and $c=\kappa_\ell$.  Define
\begin{equation*}
\begin{split}
 x_{ij\ell}&=-\varepsilon_{\tau(\ell)}
 c_{ij}^{\,\tau(\ell)},\\
 y_{ij\ell}&=-\varepsilon_{\tau(i)}
 c_{j\ell}^{\,\tau(i)},\\
 z_{ij\ell}&=-\varepsilon_{\tau(j)}
 c_{\ell i}^{\,\tau(j)}.
\end{split}
\end{equation*}
The definitions of $d_\mu$ and $K$ give
\begin{equation*}
 (d_\mu\alpha)(e_i,e_j,e_\ell)
 =cx_{ij\ell}+ay_{ij\ell}+bz_{ij\ell}.
\end{equation*}
Since $d_\mu\alpha$ is a three-form,
\begin{equation*}
 \frac16|d_\mu\alpha|^2
 =\sum_{i<j<\ell}
 \bigl(d_\mu\alpha(e_i,e_j,e_\ell)\bigr)^2.
\end{equation*}
For brevity, write $x=x_{ij\ell}$, $y=y_{ij\ell}$ and
$z=z_{ij\ell}$.  The corresponding terms in
\eqref{eq:pre-Darboux-SOS} satisfy
\begin{equation}\label{eq:triple-SOS}
\begin{split}
 &(cx+ay+bz)^2+(a^2+b^2-c^2)x^2\\
 &\quad +(b^2+c^2-a^2)y^2+(c^2+a^2-b^2)z^2\\
 &\qquad=(ax+cy)^2+(bx+cz)^2+(by+az)^2.
\end{split}
\end{equation}

It remains to determine the terms not contained in
\eqref{eq:triple-SOS}.  Given $c_{pq}^{\,k}$ with $p<q$, if
$\tau(k)\notin\{p,q\}$, then it occurs in the unique unordered triple
$\{p,q,\tau(k)\}$.  If $\tau(k)=p$, its coefficient in
\eqref{eq:pre-Darboux-SOS} is $\kappa_q^2$; if $\tau(k)=q$, its
coefficient is $\kappa_p^2$.  Therefore
\begin{equation}\label{eq:full-Darboux-SOS}
\begin{split}
 &\frac16|d_\mu\alpha|^2-2\tr(\alpha^2M_\mu)\\
 ={}&\sum_{i<j<\ell}
 \Bigl[(\kappa_i x_{ij\ell}+\kappa_\ell y_{ij\ell})^2
 +(\kappa_j x_{ij\ell}+\kappa_\ell z_{ij\ell})^2\\
 &\hspace{35mm}
 +(\kappa_j y_{ij\ell}+\kappa_i z_{ij\ell})^2\Bigr]\\
 &+\sum_{\substack{i<j,\ k\\ \tau(k)=i}}
 \kappa_j^2(c_{ij}^{\,k})^2
 +\sum_{\substack{i<j,\ k\\ \tau(k)=j}}
 \kappa_i^2(c_{ij}^{\,k})^2.
\end{split}
\end{equation}
All terms on the right-hand side of \eqref{eq:full-Darboux-SOS} are
nonnegative.  This proves
\eqref{eq:algebraic-hodge-ricci}.
\end{proof}

\begin{proof}[Proof of Theorem~\ref{thm:hodge-ricci}]
Let $\mu$ be the Lie bracket of $\nalg$.  The Ricci operator of a metric
nilpotent Lie algebra is given by
\begin{equation*}
\begin{split}
 \bar g(\Ric_\mu X,Y)={}&-\frac12\sum_i
 \bar g(\mu(X,e_i),\mu(Y,e_i))\\
 &+\frac14\sum_{i,j}
 \bar g(\mu(e_i,e_j),X)\bar g(\mu(e_i,e_j),Y).
\end{split}
\end{equation*}
Indeed, both the Killing form and the mean-curvature vector vanish; see
\cite[Corollary~7.38]{Besse1987}.  Thus $\Ric_\mu=M_\mu$, and
\eqref{eq:hodge-ricci-intro} follows from
Lemma~\ref{lem:algebraic-hodge-ricci}.  Finally,
\begin{equation*}
\begin{split}
 \frac16|d_\mu\alpha|^2-\tr(\alpha^2\Ric_\mu)
 ={}&\frac1{12}|d_\mu\alpha|^2\\
 &+\frac12\left(\frac16|d_\mu\alpha|^2
 -2\tr(\alpha^2\Ric_\mu)\right),
\end{split}
\end{equation*}
which proves \eqref{eq:hodge-ricci-coercive}.
\end{proof}

\begin{remark}
The proof of Lemma~\ref{lem:algebraic-hodge-ricci} uses only the
skew-symmetry of $\mu$ in its two input variables.  Neither the Jacobi
identity nor nilpotency is needed there.  Nilpotency is used precisely in
the identity $\Ric_\mu=M_\mu$.
\end{remark}

\section{Harmonic torsion and uniqueness}\label{sec:harmonicity}

We now combine the codifferential identity with the Hodge--Ricci
inequality.  The proof does not require a decomposition of the nilpotent
Lie algebra or a restriction on the image of $d_\mu^*$.

\begin{proof}[Proof of Theorem~\ref{thm:main}]
Let $(\nalg,\mu,\bar g,H)$ represent the generalized nilsoliton in its
preferred splitting and set $\eta=d_\mu^*H$.  By
Proposition~\ref{prop:codifferential-identity} and \eqref{eq:RicB},
\begin{equation}\label{eq:main-energy}
 0=\frac16|d_\mu\eta|^2-\tr(\eta^2\Ric_\mu)
   +\frac14\tr(\eta^2H^2).
\end{equation}
The last term is nonnegative.  Indeed, by \eqref{eq:H-square}, both
$\eta^2$ and $H^2$ are positive semidefinite, and hence
\begin{equation*}
 \tr(\eta^2H^2)
 =\tr\!\left((\eta^2)^{1/2}H^2(\eta^2)^{1/2}\right)\geq0.
\end{equation*}
Using \eqref{eq:hodge-ricci-coercive}, equation
\eqref{eq:main-energy} becomes
\begin{equation*}
\begin{split}
 0={}&\frac1{12}|d_\mu\eta|^2
 +\frac12\left(\frac16|d_\mu\eta|^2
 -2\tr(\eta^2\Ric_\mu)\right)\\
 &+\frac14\tr(\eta^2H^2).
\end{split}
\end{equation*}
Every term is nonnegative, and therefore $d_\mu\eta=0$.  Since
$\eta=d_\mu^*H$, the codifferential pairing gives
\begin{equation*}
 |\eta|^2=\bar g(d_\mu^*H,\eta)
 =\frac13\bar g(H,d_\mu\eta)=0.
\end{equation*}
Thus $d_\mu^*H=0$.  Since $d_\mu H=0$ is part of the exact Courant
algebroid data, $H$ is harmonic.
\end{proof}

\begin{proof}[Proof of Corollary~\ref{cor:uniqueness}]
By Theorem~\ref{thm:main}, the preferred torsion of every generalized
nilsoliton on $E$ is harmonic.  The assertion now follows from
\cite[Theorem~7.11]{FLS2024}.
\end{proof}

\begin{remark}
The group $K$ includes the scaling of the exact Courant
algebroid used in \cite{FLS2024}.  Thus Corollary~\ref{cor:uniqueness} is
not merely a uniqueness statement modulo $\Aut(\nalg)$, nor does its
scaling reduce to $g\mapsto cg$ on a fixed representative.
\end{remark}

\begin{remark}
The only geometric property of $\nalg$ used in the proof is the nilpotent
Ricci formula $\Ric_\mu=M_\mu$.  On a general Lie algebra, the Ricci operator
also contains the Killing-form and mean-curvature terms.  The argument
does not determine the sign of their contractions with $\eta^2$.
\end{remark}

\section{Refinements and examples}\label{sec:refinements}

The general inequality is sufficient for harmonicity.  In two important
classes, its left-hand side admits a simpler expression.  We record these
identities and conclude with an example showing that the soliton equation
is essential.

\subsection{The two-step identity}\label{subsec:two-step-identity}

Let $(\nalg,\mu,\bar g)$ be a two-step nilpotent metric Lie algebra.  Set
$\zalg=Z(\nalg)$ and $\valg=\zalg^\perp$.  For $z\in\zalg$, define
$J_z\in\mathfrak{so}(\valg)$ by
\begin{equation*}
 \bar g(J_zX,Y)=\bar g(z,\mu(X,Y)).
\end{equation*}
Choose an orthonormal basis $\{z_1,\ldots,z_r\}$ of the full center and
write $J_a=J_{z_a}$.  If $z_a\perp[\nalg,\nalg]$, then $J_a=0$.

\begin{proposition}\label{prop:two-step-refinement}
If
\begin{equation*}
 \alpha=\sum_{a=1}^r u_a^\flat\wedge z^a+\gamma,
 \qquad
 u_a\in\valg,
 \qquad
 \gamma\in\Lambda^2\zalg^*,
\end{equation*}
then
\begin{equation}\label{eq:two-step-refinement}
 \frac16|d_\mu\alpha|^2-2\tr(\alpha^2\Ric_\mu)
 =\sum_{a<b}|J_bu_a-J_au_b|^2.
\end{equation}
Consequently,
\begin{equation*}
 \frac16|d_\mu\alpha|^2-\tr(\alpha^2\Ric_\mu)
 =\frac1{12}|d_\mu\alpha|^2
 +\frac12\sum_{a<b}|J_bu_a-J_au_b|^2.
\end{equation*}
\end{proposition}

\begin{proof}
Set $\omega_a(X,Y)=\bar g(J_aX,Y)$ and
$g_{ab}=\frac12\bar g(\omega_a,\omega_b)$.  The two-step Ricci formulas
are
\begin{equation*}
 \Ric_\mu|_{\valg}=\frac12\sum_aJ_a^2,
 \qquad
 \bar g(\Ric_\mu z_a,z_b)=\frac12g_{ab},
 \qquad
 \Ric_\mu(\valg,\zalg)=0;
\end{equation*}
see \cite[Proposition~2.5]{Eberlein1994}.  Write
$\alpha_{11}=\sum_a u_a^\flat\wedge z^a$.  Since $dz^a=-\omega_a$,
\begin{equation*}
 d_\mu\alpha_{11}=\sum_a\omega_a\wedge u_a^\flat,
 \qquad
 d_\mu\gamma=-\sum_a\omega_a\wedge\iota_{z_a}\gamma.
\end{equation*}
The two terms have different bidegrees and are orthogonal.

For two-forms $\xi,\psi\in\Lambda^2\valg^*$ and $x,y\in\valg$, one has
\begin{equation*}
 \frac16\bar g(\xi\wedge x^\flat,\psi\wedge y^\flat)
 =\frac12\bar g(\xi,\psi)\bar g(x,y)
 -\bar g(\iota_y\xi,\iota_x\psi).
\end{equation*}
Using this identity and $\iota_u\omega_a=(J_au)^\flat$, we obtain
\begin{equation*}
 \frac16|d_\mu\alpha_{11}|^2
 =\sum_{a,b}\left(
 g_{ab}\bar g(u_a,u_b)
 -\bar g(J_au_b,J_bu_a)\right).
\end{equation*}
The Ricci formulas give
\begin{equation*}
 2\tr(\alpha_{11}^2\Ric_\mu)
 =-\sum_{a,b}|J_bu_a|^2
 +\sum_{a,b}g_{ab}\bar g(u_a,u_b).
\end{equation*}
Subtracting the two expressions and combining the terms indexed by
$(a,b)$ and $(b,a)$ gives
\begin{equation*}
 \frac16|d_\mu\alpha_{11}|^2-2\tr(\alpha_{11}^2\Ric_\mu)
 =\sum_{a<b}|J_bu_a-J_au_b|^2.
\end{equation*}

Set $q_a=\iota_{z_a}\gamma$.  Then
\begin{equation*}
 \frac16|d_\mu\gamma|^2
 =\sum_{a,b}g_{ab}\bar g(q_a,q_b)
 =2\tr(\gamma^2\Ric_\mu|_{\zalg}).
\end{equation*}
The mixed part of $(\alpha_{11}+\gamma)^2$ does not contribute because
$\Ric_\mu(\valg,\zalg)=0$.  Formula~\eqref{eq:two-step-refinement}
follows.
\end{proof}

\begin{remark}
With respect to the bigrading
$\Lambda^{p,q}=\Lambda^p\valg^*\wedge\Lambda^q\zalg^*$, one has
\begin{equation*}
 d_\mu=-\sum_a\omega_a\wedge\iota_{z_a},
 \qquad
 d_\mu^*=-\sum_a z^a\wedge\Lambda_a,
\end{equation*}
where, for an orthonormal basis $\{e_p\}$ of $\valg$,
\begin{equation*}
 \Lambda_a=\sum_{p<q}\omega_a(e_p,e_q)
 \iota_{e_q}\iota_{e_p}.
\end{equation*}
Hence
\begin{equation*}
 \im\bigl(d_\mu^*: \Lambda^3\nalg^*\longrightarrow
 \Lambda^2\nalg^*\bigr)
 \subset\Lambda^{1,1}\oplus\Lambda^{0,2},
\end{equation*}
which is precisely the space appearing in
Proposition~\ref{prop:two-step-refinement}.
\end{remark}

\subsection{The almost-abelian identity}\label{subsec:almost-abelian}

\begin{proposition}\label{prop:almost-abelian-refinement}
Let $(\nalg,\mu,\bar g)$ be a nilpotent metric Lie algebra with a
codimension-one abelian ideal $\aalg$.  Choose a unit vector
$e_0\in\aalg^\perp$ and set
\begin{equation*}
 \nalg=\R e_0\oplus^\perp\aalg,
 \qquad
 A=\ad_{e_0}|_{\aalg}.
\end{equation*}
For every $\alpha\in\Lambda^2\aalg^*$,
\begin{equation}\label{eq:almost-abelian-refinement}
 \frac16|d_\mu\alpha|^2-2\tr(\alpha^2\Ric_\mu)
 =\frac12|\rho(A^*)\alpha|^2.
\end{equation}
Moreover,
\begin{equation*}
 \im\bigl(d_\mu^*: \Lambda^3\nalg^*\longrightarrow
 \Lambda^2\nalg^*\bigr)\subset\Lambda^2\aalg^*.
\end{equation*}
\end{proposition}

\begin{proof}
Since $\aalg$ is abelian,
\begin{equation*}
 d_\mu\alpha=e^0\wedge\rho(A)\alpha,
 \qquad
 \Ric_\mu|_{\aalg}=\frac12[A,A^*],
 \qquad
 \Ric_\mu(e_0,\aalg)=0.
\end{equation*}
The representation on forms satisfies $\rho(A)^*=\rho(A^*)$.  By
Lemma~\ref{lem:representation-contraction},
\begin{equation*}
\begin{split}
 \frac16|d_\mu\alpha|^2-2\tr(\alpha^2\Ric_\mu)
 &=\frac12|\rho(A)\alpha|^2
 +\frac12\bar g(\rho([A,A^*])\alpha,\alpha)\\
 &=\frac12|\rho(A^*)\alpha|^2.
\end{split}
\end{equation*}
This proves \eqref{eq:almost-abelian-refinement}.

Write $H=e^0\wedge\beta+\gamma$, where
$\beta\in\Lambda^2\aalg^*$ and $\gamma\in\Lambda^3\aalg^*$.  Since
$d_\mu\gamma=e^0\wedge\rho(A)\gamma$ and
$d_\mu(e^0\wedge\beta)=0$, adjointness gives
$d_\mu^*H=\rho(A^*)\beta\in\Lambda^2\aalg^*$.
\end{proof}

\begin{remark}
The standard filiform Lie algebra
\begin{equation*}
 L_n:\qquad [E_1,E_i]=E_{i+1},
 \qquad 2\leq i\leq n-1,
\end{equation*}
has the codimension-one abelian ideal
$\operatorname{span}\{E_2,\ldots,E_n\}$.  Thus
Proposition~\ref{prop:almost-abelian-refinement} gives a direct
specialization of the general inequality to $L_n$.
\end{remark}

We finish with a simple example showing that closedness alone does not
imply coclosedness, even in the two-step case.

\begin{example}\label{ex:closed-not-coclosed}
Let $\nalg=\mathfrak h_3\oplus\R$ and choose an orthonormal coframe
$\{e^1,e^2,e^3,e^4\}$ satisfying
\begin{equation*}
 de^1=de^2=de^4=0,
 \qquad
 de^3=-ae^{12},
 \qquad a\ne0.
\end{equation*}
For $H=e^{124}$, one has
\begin{equation*}
 dH=0,
 \qquad
 d^*H=-ae^{34}\ne0.
\end{equation*}
Indeed, $d(e^{34})=-ae^{124}$, and the formula for $d^*H$ follows by
adjointness.  With the convention $\Delta=-(dd^*+d^*d)$, one also has
$\Delta H=-a^2e^{124}$.
\end{example}


\begin{thebibliography}{99}

\bibitem{Besse1987}
A.~L. Besse,
\emph{Einstein Manifolds},
Ergeb. Math. Grenzgeb. (3), vol.~10,
Springer-Verlag, Berlin, 1987.
\href{https://doi.org/10.1007/978-3-540-74311-8}
{doi:10.1007/978-3-540-74311-8}.

\bibitem{Eberlein1994}
P.~Eberlein,
\emph{Geometry of $2$-step nilpotent groups with a left invariant metric},
Ann. Sci. \`Ecole Norm. Sup. (4) \textbf{27} (1994), no.~5,
611--660.
\href{https://doi.org/10.24033/asens.1702}
{doi:10.24033/asens.1702}.

\bibitem{FLS2024}
E.~Fusi, R.~A. Lafuente, and J.~Stanfield,
\emph{The homogeneous generalized Ricci flow},
to appear in Selecta Math. (N.S.) (2026),
\href{https://arxiv.org/abs/2404.15749}
{arXiv:2404.15749v1 [math.DG]}.

\bibitem{GFS2021}
M.~Garc\'ia-Fern\'andez and J.~Streets,
\emph{Generalized Ricci Flow},
University Lecture Series, vol.~76,
American Mathematical Society, Providence, RI, 2021.
\href{https://doi.org/10.1090/ulect/076}
{doi:10.1090/ulect/076}.

\bibitem{Gualtieri2011}
M.~Gualtieri,
\emph{Generalized complex geometry},
Ann. of Math. (2) \textbf{174} (2011), no.~1, 75--123.
\href{https://doi.org/10.4007/annals.2011.174.1.3}
{doi:10.4007/annals.2011.174.1.3}.

\bibitem{Lauret2001}
J.~Lauret,
\emph{Ricci soliton homogeneous nilmanifolds},
Math. Ann. \textbf{319} (2001), no.~4, 715--733.
\href{https://doi.org/10.1007/PL00004456}
{doi:10.1007/PL00004456}.

\bibitem{Paradiso2021}
F.~Paradiso,
\emph{Generalized Ricci flow on nilpotent Lie groups},
Forum Math. \textbf{33} (2021), no.~4, 997--1014.
\href{https://doi.org/10.1515/forum-2020-0171}
{doi:10.1515/forum-2020-0171}.

\end{thebibliography}
\end{document}